\documentclass[smallcondensed]{svjour3}

\usepackage[scanall]{psfrag}
\usepackage{graphicx}
\usepackage{epstopdf}
\usepackage{tikz}
\usepackage[brazil,english]{babel}
\usepackage[latin1]{inputenc}
\usepackage{amsfonts,amsmath,amssymb}
\usepackage{mathtools}
\usepackage{multirow}
\usepackage{color}
\usepackage{here}
\usepackage{array}
\usepackage{amssymb}
\usepackage{pdfsync}
\usepackage{scrextend}
\usepackage{booktabs}
\usepackage{tabularx}
\usepackage{listings}
\usepackage{natbib}
\bibpunct[, ]{(}{)}{,}{a}{}{,}%
\def\bibhang{24pt}%
\smartqed

\usepackage{tikz}
\usetikzlibrary{matrix,shapes,arrows,positioning,chains}
\usepackage{xspace}
\usepackage{float}

\graphicspath{ {./Figures/} }

\newcommand{\lV}{\vspace{-1mm}}

\newcommand{\lH}{\hspace{-1mm}}

\newcommand{\llH}{\hspace{-2mm}}

\newcommand{\T}[1]{\text{#1}}
\newcommand{\mi}{\T{--}}

\begin{document}

\title{Derivation and Generation of Path-Based Valid Inequalities for Transmission Expansion Planning}

\author{J. Kyle Skolfield \and Laura M. Escobar \and Adolfo R. Escobedo}

\institute{J. Kyle Skolfield \at School of Computing, Informatics, and Decision Systems Engineering (CIDSE), Arizona State University, Tempe, Arizona \\\email{kyle.skolfield@asu.edu} \and Laura M. Escobar \at Electrical Engineering Department, S\~ao Paulo State University (UNESP), Ilha Solteira, S\~ao Paulo, Brazil \\\email{lauramonicaesva@gmail.com} \and Adolfo R. Escobedo \at School of Computing, Informatics, and Decision Systems Engineering (CIDSE), Arizona State University, Tempe, Arizona, \\\email{adRes@asu.edu}}

\titlerunning{Path-based Valid Inequalities for Transmission Expansion Planning}

\maketitle

\begin{abstract}
This paper seeks to solve the long-term transmission expansion planning problem in power systems more effectively by reducing the solution search space and the computational effort. The proposed methodology finds and adds cutting planes based on structural insights about bus angle-differences along paths.  Two lemmas and a  theorem are proposed which formally establish the validity of these cutting planes onto the underlying mathematical formulations.  The path-based bus angle-difference constraints, which tighten the relaxed feasible region, are used in combination with branch-and-bound to find lower bounds on the optimal investment of the transmission expansion planning problem. This work also creates an algorithm that automates the process of finding and applying the most effective valid inequalities, resulting in significantly reduced testing and computation time.  The algorithm is implemented in Python, using CPLEX to add constraints and solve the exact DCOPF-based transmission expansion problem.  This paper uses two different-sized systems to illustrate the effectiveness of the proposed framework: a modified IEEE 118-bus system and a modified Polish 2383-bus system.
\end{abstract}


\acknowledgement{Laura M. Escobar's work is supported by the Brazilian institutions CAPES, CNPq (Grant NO. 142150/2015-0) and S\~ao Paulo Research Foundation--FAPESP (Grant NO. 2015/21972-6). The authors acknowledge Research Computing at Arizona State University for providing HPC resources that have contributed to the research results reported within this paper. URL: http://www.researchcomputing.asu.edu}

\section*{Nomenclature} 

\noindent \emph{Sets }:
\begin{labeling}{$g\in G$\hspace{20pt}}
\item[$n\in B$]            Buses (i.e., nodes)
\item[$(i,j)\in \Omega$] 			 Corridors (i.e., arcs)
\end{labeling}

\noindent \emph{Parameters}:
 \begin{labeling}{$\omega^{\max}_{ij}$\hspace{25pt}}
    \item[$c_{ij,k}$ ]                Cost of line $k$ in corridor $(i,j)$ 
    \item[$c_n$]                        Per unit cost of generation at bus $n$
    \item[$\omega^{0}_{ij}$]               Number of established lines in corridor $(i,j)$ 
    \item[$\overline{\omega}_{ij}$ ]     Maximum number of expansion lines in corridor $(i,j)$ 
    \item[$\overline{g}_{n}$]      Maximum limit of power generation at bus $n$
    \item[${d}_{n}$]               Active power demand at bus $n$
    \item[$\overline{\theta}_{ij}$]  Maximum bus angle-difference magnitude
    \item[$\overline{P}_{ij,k}$]        Capacity of candidate line $k$ in corridor $(i,j)$
    \item[$\overline{P}^0_{ij,k}$]      Capacity of existing line $k$ in corridor $(i,j)$
    \item[$x_{ij,k}$]                  Reactance of line $k$ in corridor $(i,j)$
    \item[$b_{ij,k}$]                 Susceptance of line $k$ in corridor $(i,j)$
    \item[$M_{ij}$]                       Large number (big-\textit{M}) used in the disjunctive constraints 
    \item[$\sigma$]                 Scaling factor to align generation and expansion costs
\end{labeling}

\noindent \emph{Continuous Variables}:
 \begin{labeling}{$P_{g}^{\max}, P_{g}^{\min}$\hspace{1pt}}
    \item[$P_{ij,k}^{0}$]  Active power flow in existing line $k$ in corridor $(i,j)$
    \item[$P_{ij,k}$]  Active power flow in candidate line $k$ in corridor $(i,j)$
    \item[$g_{n}$]  Active power output of generator in bus $n$
    \item[$\theta_{n}$]         Voltage angle at bus $n$
\end{labeling}

\noindent \emph{Binary Variables}:
\vspace{0pt}	
 \begin{labeling}{$P_{g}^{\max}, P_{g}^{\min}$\hspace{1pt}}
    \item[$y_{ij,k}$]  Decision to construct the $k^{th}$ candidate line in corridor $(i,j)$
\end{labeling}
\section{Introduction} \label{introduction}

\subsection{Background}

The objective of the Transmission-network Expansion Planning (TEP) problem is to find the least costly investment options in new transmission devices required to ensure proper power system operations into the future \citep{garv}. Optimizing this problem is important because the transmission network belongs to the so-called heavy technologies, which are both expensive and difficult to withdraw or relocate once they are installed \citep{Domtesis}. Inadequate long-term planning can lead to low service quality, excessive oversizing, inefficient  systems with high operating costs, and delays in the expansion of electricity markets. While new systems are growing in size and the demands imposed on them are increasing, deregulation and other challenges have made meeting those requirements ever more difficult \citep{lum16new}. Hence, it is critical to obtain solutions that maximize cost efficiency to enable the incorporation of more avant-garde technologies into the smart grid. For these reasons, it is necessary to devise new planning methodologies that can effectively deal with the associated combinatorial difficulties of the underlying TEP optimization models. 

In its standard form, TEP consists of linear and non-linear functions that include continuous variables (e.g., voltage angles, power flows, etc.) and integer variables (decisions to, e.g., add lines to the network). TEP can be formulated as a non-convex, mixed-integer nonlinear programming problem. It is NP-hard, which makes its solution generally intractable \citep{latorre}. This is exacerbated by the fact that in large-scale systems, the number of network components and associated restrictions can number in the hundreds or thousands. That is, the size and/or topology of the transmission network and the inclusion of discrete variables for representing possible transmission investments lead to a combinatorial explosion of potential solutions. Due to these complications, TEP cannot practically be solved using standard optimization techniques, in general. Different modeling techniques and algorithms have been proposed to expedite solution times (e.g., \cite{Zeng,Silva,Cabrera,Choi,TAP}). Exact methods require larger calculation times when compared to those required by metaheuristic techniques such as Tabu Search \citep{Gallego,Garcia2015} and Genetic Algorithms \citep{Gallego,Oliviera}, among others. However, the latter techniques generally do not provide formal optimality guarantees. In small- and medium-sized systems, the ideal solution can be found using methods such as branch-and-bound or branch-and-cut when a disjunctive integer linear programming model approximation is utilized \citep{Bahi,souasa,Di}.  Such methods provide formal guarantees, but they are demanding computationally.  They also include decomposition techniques, such as hierarchical Benders decomposition (e.g. \cite{Romero1994, bin01new, Haffner2001}).   Additionally, recent work has used Benders decomposition techniques to solve generation and transmission expansion planning together \citep{Jenabi2015}.  The valid inequalities presented in this paper can be seen as a complementary technique for solution time reduction to these exact methods.  
\lV

\subsection{Aim and Contributions}

This work considers a DCOPF-based mixed-integer programming version of the \textit{static} TEP problem which consists of a single investment period occurring at the beginning of the planning horizon and is a subproblem of the dynamic TEP problem. The choice of this model helps illustrate the computational intractability of TEP even for this basic context and is useful for various practical studies.  Moreover, it highlights the potential of the fundamental insights introduced herein to be extended to a variety of more complex TEP models with a similar core structure (e.g. \cite{bin01new,plo17ope,Vinasco}, etc.). Explicitly, this work derives and implements a set of theoretical contributions for detecting and including structural information on the underlying network which is relevant to any DCOPF-based model that incorporates the linear relationship between bus angle-differences and power flows (i.e., ``$B-\theta$" constraints) into the constraint set.  The insights presented in this paper may be applied to aid in solving a variety of problem classes, since this structural information is common to many power system formulations.  Such insights are captured via the concept of valid inequalities, which represent one of the most effective exact solution techniques and are a highly active research area in mathematical programming \citep{conforti2014integer}. 

Other papers have explored structural insights based on bus angle-differences, which serve as the inspiration of this work.  In particular, in \cite{Laura1}, a subset of the classes of valid inequalities introduced in this paper were applied in an ad hoc manner, in particular, only those from the herein included lemmas, which are proved in the present paper for the first time. Moreover, while the lemmas are helpful in providing insight for the major theorems derived in this paper, and their implementation could produce coincidental improvement in CPLEX due to the ordering of constraints, it can be shown that it is analytically impossible for them to reduce the linear relaxation solution space of TEP. This is because the valid inequalities presented in these lemmas are obtained as linear combinations of the original set of constraints. In short, the cited work lacked the systematic and theoretical depth featured in this work from an operations research perspective. The present paper automates and extends the process of that work.  Specifically, it formally establishes the validity of two classes of valid inequality used therein via two lemmas, plus one additional class, proved via a theorem, which can in fact reduce the solution space of the linear relaxation.

In addition to these theoretical contributions, this work also provides techniques for applying the theory in the form of a heuristic algorithm used to help find the more effective candidate valid inequalities (also referred to herein as cuts).  These techniques are then used to perform computational experiments that show the effectiveness of the proposed valid inequalities in reducing the solution time of two modified benchmark instances.  While their effectiveness is shown herein for static TEP, the reduction in solution time would be amplified in, for example, stochastic programming approaches to TEP. In these approaches, many scenarios need to be solved with each using the same collection of valid inequalities, since the first-stage decisions usually involve the structure of the network. A similar argument holds for solving the multi-period TEP.

The structure of the paper is as follows:  Section 2 introduces the disjunctive model used for modeling TEP.  Section 3 presents the key insights and intuition for deriving and generating the valid inequalities. Section 4 contains the main contribution of this work, the lemmas and theorem which prove the validity of the discussed cuts.  Section 5 presents numerical results from testing the application of these theorems to three different test cases, and Section 6 summarizes the conclusions drawn from these results.

\section{Modeling Framework}

The nonlinear ACOPF model for TEP can be transformed into a mixed-integer linear model with bilinear equations \citep{zhang2013transmission}. This model is itself transformed into a disjunctive model with binary variables, which is always possible using a large enough disjunctive coefficient (big-\textit{M}). In the disjunctive model, a binary variable is considered for each candidate line, which converts the original mixed-integer non-linear program into a mixed-integer linear program (MILP).  The DCOPF-based model is appropriate for TEP.  First, it is widely used in industrial practice, especially for planning purposes \citep{kocuk2016cycle}.  Additionally, this approach is the most common classical optimization approach in the literature \citep{lum16new}.  Finally, long-term planning is primarily concerned with active power rather than reactive power, and consequently the assumption in DCOPF that active power is much larger than reactive is reasonable.  The main concerns that are only captured with an AC model (e.g., stability of the network) can be incorporated in a more short-term, operational perspective \citep{lumbreras2014automatic}.  The full model is as follows.

The objective function \eqref{eq_disyuntivo_r_1} is to minimize the joint cost of generation and investments in new lines, with investment considered to be performed at the beginning of the planning horizon:
\begin{align}
& \hspace{0pt}\displaystyle  \min  \sum_{(i,j) \in \Omega}\sum_{k=1}^{\bar{\omega}_{ij}} c_{ij,k}y_{ij,k} + \sum_{n \in B} \sigma c_n g_n .& \label{eq_disyuntivo_r_1} 
\end{align}
Here, $c_{ij,k}$ is the cost of each line in corridor $(i,j)$ and binary variable $y_{ij,k}$ represents the decision to add the $k^{th}$ candidate line in corridor $(i,j)$.  When $y_{ij,k}= 1$, the $k$th candidate line is added in corridor $(i,j)$.  Additionally, $\bar{\omega}_{ij}$ is the maximum number of candidate lines considered in corridor $(i,j)$, and $\Omega$ is the set of expansion corridors in the expansion plan.  Finally, note that the generation costs in the objective function are weighted by a factor $\sigma$ to make generation costs and planning costs comparable \citep{min18sol}.  The set of constraints is as follows:
\lV
\begin{flalign}	
& \hspace{0pt}\displaystyle  \mathrlap{\sum_{(n,i)\in\Omega} \left(\sum_{k=1}^{\omega_{ij}^0}P_{ni,k}^{0} + \sum_{k=1}^{\bar{\omega}_{ij}} P_{ni,k} \right) -  \sum_{(i,n)\in\Omega} \left(\sum_{k=1}^{\omega_{ij}^0}P_{in}^{0} + \sum_{k=1}^{\bar{\omega}_{ij}} P_{in,k} \right) + g_{n} = d_{n}} &\forall n \in B \label{eq_disyuntivo_r_2} \\
& \hspace{0pt}\displaystyle - \overline{P}^0_{ij,k} \leq P_{ij,k}^{0} \leq \overline{P}^0_{ij,k} &\forall(i,j) \in \Omega, k \in \{1 \dots \omega_{ij}^0\}  \label{eq_disyuntivo_r_3}  \\
& \hspace{0pt}\displaystyle  - \overline{P}_{ij,k} y_{ij,k} \leq P_{ij,k} \leq \overline{P}_{ij,k} y_{ij,k} &\forall(i,j) \in \Omega , k \in  \{1 \dots \bar{\omega}_{ij}\}  \label{eq_disyuntivo_r_4} \\
& \hspace{0pt}\displaystyle  \frac{\mi 1}{b_{ij,k}}P_{ij,k}^{0} - (\theta_{i}-\theta_{j})= 0 &\forall(i,j) \in \Omega, k \in \{1 \dots \omega_{ij}^0 \} \label{eq_disyuntivo_r_5} \\
& \hspace{0pt}\displaystyle \mathrlap{ \mi M_{ij}(1-y_{ij,k}) \leq \frac{\mi 1}{b_{ij,k}}P_{ij,k}-(\theta_{i}-\theta_{j}) \leq M_{ij}(1-y_{ij,k})} &\forall(i,j) \in \Omega , \; k \in \{ 1 \dots \bar{\omega}_{ij}\}\label{eq_disyuntivo_r_6}  \\
& \hspace{0pt}\displaystyle  g_{n} \leq \overline{g}_n   &\forall n \in B & \label{eq_disyuntivo_r_9} \\
& \hspace{0pt}\displaystyle  -\overline{\theta} \leq \theta_{i}-\theta_{j} \leq \overline{\theta} &\forall (i,j) \in \Omega  \label{eq_disyuntivo_r_10a} \\
& \hspace{0pt}\displaystyle  y_{ij,k}  \in  \left\{  0,1  \right\}  &\forall(i,j) \in \Omega ,  k \in \{ 1 \dots \bar{\omega}_{ij}\}   \label{eq_disyuntivo_r_11} \\
& \hspace{0pt}\displaystyle  g_n \geq 0, \; \theta_{n} \; \mbox{unr.}  &\forall n \in B \label{eq_disyuntivo_r_12} \\
& \hspace{0pt}\displaystyle P_{ij,k}^{0},P_{ij,k} \; \mbox{unr.} &\forall (i,j) \in \Omega, k \in \{1 \dots \bar{\omega}_{ij}\} 
\label{eq_disyuntivo_r_13}
\end{flalign}

The constraints start with \eqref{eq_disyuntivo_r_2}, which interrelates the active power flows that arrive at and leave bus $n$ through both existing and candidate lines and the demand and supply of active power at bus $n$.  \eqref{eq_disyuntivo_r_3} represents the limit of active power flow through the current network in corridor $(i,j)$, where $P_{ij,k}^{0}$ is the power flow in the $k^{th}$ existing line.  \eqref{eq_disyuntivo_r_4} represents the limit of active power flow through the candidate lines in corridor $(i,j)$, while \eqref{eq_disyuntivo_r_5} and \eqref{eq_disyuntivo_r_6} show the link between the active power flows of a corridor $(i,j)$ and the bus angle-difference between incident buses $i$ and $j$. Equations \eqref{eq_disyuntivo_r_5} and \eqref{eq_disyuntivo_r_6} both represent Kirchhoff's second law, either for each existing line or each candidate line to be added to the transmission system, respectively.   \eqref{eq_disyuntivo_r_6} becomes active when the decision variable $y_{ij,k}$ takes the value of 1, i.e. when that candidate line is built. Otherwise, a sufficiently large big-\textit{M} parameter $M_{ij}$ ensures that \eqref{eq_disyuntivo_r_6} is extraneous for the model.  Finding the best value for $M_{ij}$ is a shortest path problem in connected networks, but a longest path problem in disconnected networks, which is itself an NP-hard problem \citep{bin01new}.  Because power networks are generally connected including in the instances tested herein, except for perhaps a handful of considered buses, these problems are solved while pre-processing the networks, in order to use the best possible big-\textit{M} parameter for each pair of buses.   \eqref{eq_disyuntivo_r_9} presents the limits of the active power supply for the generators, where a bus $n$ with no generator is assumed to have $\bar{g}_n =0$.  
\eqref{eq_disyuntivo_r_10a} enforces the maximum bus angle-difference for adjacent bus-pairs $(i,j) \in \Omega$, i.e. those bus-pairs connected by a corridor. Finally, \eqref{eq_disyuntivo_r_11}, \eqref{eq_disyuntivo_r_12} and \eqref{eq_disyuntivo_r_13} give the variable domains. 

\section{Motivating the Derivation and Generation of Path-based Valid Inequalities}

Due to the combinatorial explosion of TEP, it is not possible to find an optimal solution for large-scale systems using standard, off-the-shelf algorithms. The computational difficulty of the problem is related directly to the size of the system to be analyzed.  However, other factors increase computational difficulty, including the connectivity of the buses or how well the system is enmeshed.  This is complicated by the ``Braess Paradox,'' according to which a more inefficient system can be obtained when adding lines to the transmission system \citep{neil}.  

To solve NP-hard problems, it is often useful to investigate the structural characteristics of a particular instance. This knowledge can be highly valuable when it comes to designing effective exact solution methods \citep{nem88int,conforti2014integer}.  One key application of this knowledge is to derive valid inequalities (VIs): additional problem constraints that preserve the original solution space $\mathcal{P}$ but may otherwise reduce an associated relaxed solution space $\mathcal{P}^R\subseteq\mathbb{R}^{n}$, where $\mathcal{P}\subset\mathcal{P}^R$.  Formally, for the set $\mathcal{P} \subset \mathbb{R}^n$, the coefficient vector ${\boldsymbol{\pi}}=(\pi_1,\dots,\pi_n)\in\mathbb{R}^{n}$, and the constant $\pi_0 \in \mathbb{R}$, the inequality $\boldsymbol{\pi}\mathbf{y} \leq \mathbf{\pi}_0$ is called a \textit{valid inequality} for $\mathcal{P}$ if it is satisfied by all points $\mathbf{y} \in \mathcal{P}$ (i.e., herein, $\mathcal{P}$  is the TEP solution space). Because the solution of MILP typically proceeds by solving a sequence of linear relaxations, adding structurally useful VIs as cutting planes can reduce the number of such linear problems solved in a branch-and-bound framework, thus decreasing the computational time necessary to solve the overall problem \citep{nem88int}. The proposed method seeks to provide mechanisms that reduce the size of the solution space by incorporating structural information of TEP that can eliminate unpromising settings of decision variables. 

The structural insights derived in this work stem from the relationships between the bus angle and flow decision variables that characterize DCOPF-based transmission system models. Specifically, if there is an existing line with index $k$ in corridor $(i,j)\in\Omega$, with $i<j$, an \textit{angular VI} relating the difference between $\theta_i$ and $\theta_j$ can be obtained through $P_{ij,k}$ (the flow along the line), as follows:
\begin{equation}
\theta_i-\theta_j=\frac{-1}{b_{ij,k}}P_{ij,k}=x_{ij,k}P_{ij,k}\le x_{ij,k}\bar P_{ij,k},\label{eq_busAng_diff}
\end{equation}
where $x_{ij}$ and $\bar P_{ij,k}$ are the line reactance and flow capacity, respectively. The right hand side of this inequality is referred to henceforth as a \textit{capacity-reactance product} and may be useful for improving angular VIs as presented here.  Note that \eqref{eq_busAng_diff} is a direct result of \eqref{eq_disyuntivo_r_4}-\eqref{eq_disyuntivo_r_6}. The present work leverages such adjacent-bus VIs to derive formal restrictions on non-adjacent buses and on buses connected via multiple parallel paths in the network. That is, the TEP model (and the DCOPF model, generally) provides only simple angular constraints for the buses that are directly connected via a transmission line. However, by forming a single path connecting adjacent buses in the transmission network, these VIs can be combined into potentially tighter \textit{path-based} constraints relating the initial bus angle and the terminating bus angle of said path and the corresponding flow restrictions of each corridor along the path. Even stronger restrictions may be obtained from the combination of VIs along parallel paths---two otherwise disjoint paths which share initial and terminating buses---by taking the tighter of the separate bus angle-difference restrictions or, equivalently, flow restrictions. An example application of these insights is illustrated in Figure \ref{fig:fig1} via a stylized bus-line diagram consisting of bus set $B=\{i_0,i_1,i_2\}$, corridor set $\Omega=\{(i_0,i_1),(i_0,i_2),(i_1,i_2)\}$, and single lines between each pair of buses with reactances $x_{i_0,i_1}=x_{i_1,i_2}=x, x_{i_0,i_2}=3x$ and capacities $\bar P_{i_0,i_1}=\bar P_{i_1,i_2}=\bar P_{i_0,i_2}=\bar P$.  For this simple example, and for all future numerical examples, we assume that there can be at most one existing line and at most one candidate line per corridor. This allows us to increase visual clarity by dropping the third index of each variable.

\begin{figure}[h]
\centering 
\begin{tikzpicture}
\draw [very thick] (0,0) coordinate  -- (0,1) coordinate ;
\draw [very thick] (6,0) coordinate -- (6,1) coordinate;
\draw [very thick] (3,-2) coordinate  --  (3,-1.1) coordinate;
\draw (0,0.6) coordinate --(6,0.6) coordinate;
\draw (0,0.3) coordinate --(0.6,0.3) coordinate;
\draw (5.4,0.3) coordinate --(6,0.3) coordinate;
\draw (2.4,-1.5) coordinate --(3,-1.5) coordinate;
\draw (0.6,0.3) coordinate --(2.4,-1.5) coordinate;
\draw (3,-1.5) coordinate --(3.6,-1.5) coordinate;
\draw (3.6,-1.5) coordinate --(5.4,0.3) coordinate;

\node at (-0.3,0.9)  (i0) {$i_0$};

\node [right] at (6,0.9) {$i_2$};
\node [right] at (3,-1.9) {$i_1$};
\node [above] at (3,-2.85)[rectangle,draw] {VI'$\lH_{0\mi2}:\theta_{i_0}\lH-\theta_{i_2}\le 2x\bar P$};
\node [above] at (3,0.7) {VI$_{0\mi2}:\theta_{i_0}\lH-\theta_{i_2}\le 3x\bar P$};
\node [left] at (1.75,-1.2) {VI$_{0\mi1}:\theta_{i_0}\lH-\theta_{i_1}\le x\bar P\llH$};
\node [right] at (4,-1.2) {VI$_{1\mi2}:\theta_{i_1}\lH-\theta_{i_2}\le x\bar P$};
\end{tikzpicture}
\caption{The two path-based VIs adjacent to lines ($i_0,i_1$) and ($i_1,i_2$) can be combined to create the bottom boxed path-based VI, which tightens the path-based VI atop line ($i_0,i_2$).}
\label{fig:fig1}
\end{figure}
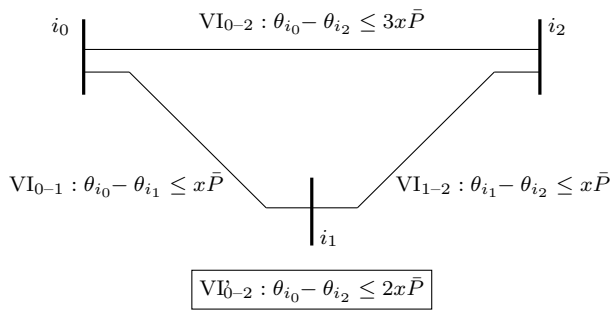

In Figure \ref{fig:fig1}, three path-based VIs (adjacent to each transmission line) are obtained by considering the capacity-reactance products of every pair of buses in the network (see \eqref{eq_busAng_diff}). Moreover, by combining two of these VIs, a tighter VI for bus angles $\theta_{i_0}$ and $\theta_{i_2}$ is obtained (see the boxed expression). It is important to remark that this constraint would be valid even in the absence of a direct transmission line between $\theta_{i_0}$ and $\theta_{i_2}$, i.e. if it were an expansion corridor. In larger networks, many such VIs can be constructed, which may or may not tighten the model's simple bus angle-difference constraints. In electric systems with high mesh levels, the number of parallel paths can increase exponentially, depending on the specific network properties \citep{kav09cyc}. Consequently, it may be prohibitive to identify and verify the strength of each possible VI for large-size systems. Instead, this work will identify the most effective of these constraints and provide data-driven insights through the use of relaxation models that are easier to solve.

We use the above ideas to generate a set of structurally useful VIs based on single paths and parallel paths that may appear in the solution to TEP.  To this end, we make use of three relaxed models.  By solving a subset of these models, each of which takes significantly less time to solve than the full MILP, we can generate a set of \textit{structural backbones}. These are solution patterns that suggest single paths and parallel paths that are more likely to occur than others in the solution to the original problem. In particular, we consider adding a VI based on any single path or collection of parallel paths which flows in the same direction in the solution to each of a combination of relaxation models.  The technique of using these relaxation models in this way will be denoted the \textit{low-effort heuristic}, first implemented in a non-algorithmic way in \cite{Laura1}. Three models are used: the linear model, where the restriction on the binary variables $y_{ij,k}$ is relaxed, allowing them to be continuous within the interval $[0,1]$; the transportation model, where the restriction that flows on all lines obey \eqref{eq_disyuntivo_r_5} and \eqref{eq_disyuntivo_r_6} is relaxed; and the hybrid model, which is similar to the transportation model, but in which only \eqref{eq_disyuntivo_r_6} is relaxed.

\section{Path-based Angular Valid Inequalities Derivation and Theorems}

\begin{figure}[H]
\centering
\begin{tikzpicture}[>=stealth',shorten >=1pt,auto,node distance=2cm, thick,main node/.style={circle,draw,font=\bfseries}]
\node at (3,2.3) [circle,draw] (i5) {$i_4$};
\node at (5,2.6) [circle,draw] (i6) {$i_5$};
\node at (7,3) [circle,draw] (i7) {$i_6$};
\node at (2,3.2) [circle,draw] (i0) {$i_0$};
\node at (3.2,4.2) [circle,draw] (i1) {$i_1$};
\node at (5.6,4.8) [circle,draw] (i2) {$i_2$};
\node at (8.8,3.7) [circle,draw] (i3) {$i_3$};
\draw [thick] (i0) -- (i1);
\draw [thick] (i0) -- (i5);
\draw [thick] (i1) -- (i2);
\draw [dashed] (i1) -- (i6);
\draw [dashed] (i2) -- (i3);
\draw [thick] (i2) -- (i6);
\draw [thick] (i5) -- (i6);
\draw [dashed] (3.51,4.4) -- (5.3,4.88);
\draw [thick] (i6) -- (i7);
\draw [thick] (i7) -- (i3);
\end{tikzpicture}
\caption{Toy Network Used to Illustrate Theorems}
\label{fig:fig2}
\end{figure}
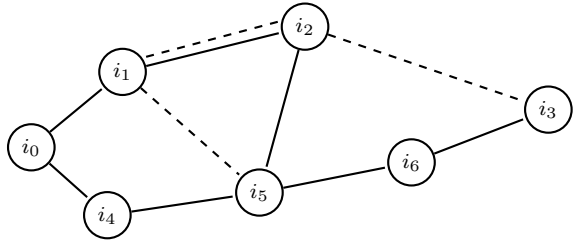

This section will introduce the theorems which are the fundamental contribution of this work. For this purpose, a graph with candidate lines (dotted edges) and existing lines (solid edges) is presented in Figure \ref{fig:fig2}. An example of these lines can be seen between buses $i_1$ and $i_2$, where there is one candidate line and one existing line. This graph will be used to illustrate an application of each lemma and theorem.

We say then that $(i,j)$ is an \textit{established corridor} of $G$ if $\omega^0_{i,j}>0$; otherwise we say that $(i,j)$ is an \textit{expansion corridor}.  To better clarify instances when we must distinguish individual lines within each corridor along a path, we introduce the vector $\hat{k}_\rho = \left< k_{i_0i_1}, \dots , k_{i_{|\rho|-1}i_{|\rho|}}\right> \subseteq \left< \{1, \dots, \omega^0_{i_0i_1} \}, \dots , \{ 1, \dots, \omega^0_{i_{|\rho|-1}i_{|\rho|}} \} \right> $ to denote any vector of valid line-indices $k_{ij}$ within each established corridor $(i,j)$ along a path $\rho$. Then, for ease of presentation, we refer to $x_{ij,k}$, where $k$ encapsulates a valid setting of element $ij$ of vector $\hat{k}$, i.e. $k\in \{1, \dots, \omega_{ij}^0\}$.  Thus, in each upcoming lemma and proof, whenever $k$ is used as a line index, it is shorthand for $k_{ij}$ when there is no ambiguity.  
Additionally, because these problems traditionally specify corridors from a lower index bus to a higher index bus, we define $\tilde{P}_{ij,k} = sgn(j - i) \cdot P_{ij,k}$, where $\textit{sgn}(i-j) = 1$ if $i > j$ and $\textit{sgn}(i-j) =-1$ if $i < j$.  Define $\tilde{P}^0_{ij,k}$ analogously for $P^0_{ij,k}$.

\subsection{Single Path over Established Corridors}
\begin{lemma}
Let $\rho=(i_0,i_1),\dots,(i_{\left|\rho\right|-1},i_{\left|\rho\right|})$ represent a directed path over established corridors in $G$. For $(i,j) \in \rho$, set coefficient vector $\boldsymbol{\pi}=(\pi_0,\pi_1,\dots,\pi_{\left|\rho\right|})\in\mathbb{R}^{\left|\rho\right|+1}$ as,
\begin{equation}\label{eqn:pi}
\pi_j= 
\begin{cases}
 \sum_{(i,m) \in \rho}x_{im,k}\cdot \overline{P}^0_{im,k} , & \scriptstyle \textit{if } j=0\\
 \textit{sgn}(i-j)x_{ij,k} , & \scriptstyle \textit{otherwise}
\end{cases},
\end{equation} 
where $k \in \{1, \dots , \omega_{ij}^0\}$ is fixed for each corridor $(i,j)$, but may vary between corridors.  Then the following two-sided inequality is valid for TEP for any $\hat{k}_\rho$:
\begin{flalign}
-\pi_0\le\sum_{(i,j) \in \rho}\pi_j \tilde{P}^0_{ij,k} \le\pi_0.\label{eqn:VI-1}
\end{flalign}
\end{lemma}
\begin{proof}  According to (\ref{eq_disyuntivo_r_5}), the flow along any fixed, existing line $k$ of corridor $(i,j)$ is given by

\begin{equation}
\tilde{P}_{ij,k}^0 = sgn(i - j)b_{ij,k}(\theta_{i} - \theta_{j}), \label{eqn:flow-1}
\end{equation}
or equivalently,

\begin{equation}
(\theta_{i} - \theta_{j}) = \pi_j \tilde{P}_{ij,k}^0 , \label{eqn:flow-1_react}
\end{equation}
where $(i,j) \in \rho$. Hence, the bus angle-difference for consecutive bus-pairs $(i_{0}, i_1)$, $(i_1,i_2)$, $(i_2,i_3)$, $\dots$, $(i_{|\rho|-1},i_{|\rho|})$ in $\rho$ can be written as: 
\begin{flalign*}
{\theta _{{i_1}}} - {\theta _{{i_0}}} &= sgn\left( {{i_0} - {i_1}} \right)  x_{{i_0 i_1},k} \tilde{P}_{{i_0i_1},k}^0 = \pi_1 \tilde{P}_{{i_0i_1},k}^0,\\
{\theta _{{i_2}}} - {\theta _{{i_1}}} &= sgn\left( {{i_1} - {i_2}} \right) x_{{i_1}{i_2},k} \tilde{P}_{{i_1}{i_2},k}^0 = \pi_2 \tilde{P}_{{i_1}{i_2},k}^0,\\
&\;\;\;\;\;\;\;\;\;\;\;\;\;\;\;\;\;\;\;\;\; \vdots \\
{\theta _{{i_{|\rho |}}}} - {\theta _{{i_{|\rho | - 1}}}} &= sgn\left( {{i_{|\rho |-1}} - {i_{|\rho |}}} \right)  x_{{i_{|\rho | - 1}}{i_{|\rho |}},k}  \tilde{P}_{{i_{|\rho | - 1}}{i_{|\rho |}}}^0 \\
\;\;\;\;\;\;\;\;\;\;\;\;\;\;\;\;\;\;\;\;\, &= {\pi_{| \rho |}} \tilde{P}_{{i_{|\rho | - 1}}{i_{|\rho |}},k}^0 .
\end{flalign*}
\vspace{3pt}
 When these equations are summed, this creates a telescoping effect on the left-hand side, which yields the following bus angle-difference equation for the starting and ending buses in $\rho$:
 \lV
\begin{flalign}
  \theta_{i_{|\rho|}}-\theta_{i_0} &=\sum_{(i,j) \in \rho} \pi_j {\tilde{P}}^0_{ij,k}\label{eqn:VI-E-pt1}\\
  &\le\sum_{(i,j) \in \rho}\left|\pi_j {\tilde{P}}^0_{ij,k}\right|\label{eqn:VI-E-pt2} \\ 
  &\le\sum_{(i,j) \in \rho} x_{ij,k}  \overline{P}^0_{ij,k} = \pi_0, \label{eqn:VI-E-pt3}
\end{flalign}
where the latter inequality is obtained by adding the rightmost inequalities from  \eqref{eq_disyuntivo_r_3}. By a similar argument we have that,
\begin{flalign}
\sum_{(i,j) \in \rho}\pi_j {\tilde{P}}^0_{ij,k} &\ge  - \sum_{(i,j) \in \rho}\left|\pi_j {\tilde{P}}^0_{ij,k}\right|\\
&\ge - \sum_{(i,j) \in \rho} x_{ij,k}  \overline{P}^0_{ij,k}\\
&=  - {\pi _0}.\label{eqn:VI-E-pt4}
\end{flalign}

Since every corridor considered has at least one existing line to select and fix as $k$, and \eqref{eqn:flow-1} holds for any line in corridor $(i,j)$, we have established the validity of \eqref{eqn:VI-1}. ${\qed}$
\end{proof}

As an example using Figure \ref{fig:fig2} the path $\rho^2:=(i_0, i_4), (i_4, i_5)$ creates the example type 1 two-sided VI:\vspace{3pt}

$\begin{array}{l}
-{{\overline{P}}_{{i_0},{i_4}}}\,{x_{{i_0},{i_4}}} - {{\overline{P}}_{{i_4},{i_5}}}\,{x_{{i_4},{i_5}}}\\
 \le {P_{{i_0},{i_1}}}\,{x_{{i_0},{i_1}}} + {P_{{i_1},{i_2}}}\,{x_{{i_1},{i_2}}}  \\
 \le \,{{\overline{P}}_{{i_0},{i_4}}}\,{x_{{i_0},{i_4}}} + {{\overline{P}}_{{i_4},{i_5}}}\,{x_{{i_4},{i_5}}} 
\end{array}$

On the same note, in Figure \ref{fig:fig2}, the path $\rho^1:=(i_0, i_1), (i_1, i_2), (i_2,i_5)$ is an established path, which creates the example type 1 two-sided VI:

\vspace{3pt}
$\begin{array}{l}
-{{\overline{P}}_{{i_0},{i_1}}}\,{x_{{i_0},{i_1}}} - {{\overline{P}}_{{i_1},{i_2}}}\,{x_{{i_1},{i_2}}} - {{\overline{P}}_{{i_2},{i_5}}}\,{x_{{i_2},{i_5}}} \\
\le {P_{{i_0},{i_1}}}\,{x_{{i_0},{i_1}}} + {P_{{i_1},{i_2}}}\,{x_{{i_1},{i_2}}} + {P_{{i_2},{i_5}}}\,{x_{{i_2},{i_5}}} \\
 \le \,{{\overline{P}}_{{i_0},{i_1}}}\,{x_{{i_0},{i_1}}} + {{\overline{P}}_{{i_1},{i_2}}}\,{x_{{i_1},{i_2}}} + {{\overline{P}}_{{i_2},{i_5}}}\,{x_{{i_2},{i_5}}} 
\end{array}$

\subsection{Parallel Paths over Established Corridors}
\begin{lemma}
Let $\rho^1,\dots,\rho^m$ represent $m>1$ alternative directed paths over established corridors in $G$ with the same starting/ending buses but with non-overlapping intermediate buses; that is, $i^r_0=i^{r'}_0$, $i^r_{|\rho^r|}=i^{r'}_{|\rho^{r'}|}$, and  $\{i^r_k\}^{|\rho^r|-1}_{k=1}\cap\{i^{r'}_k\}^{|\rho^{r'}|-1}_{k=1}=\emptyset$ for $1\le r,r'\le m$ with $r\ne r'$. Setting coefficient vectors ${\boldsymbol{\pi}^r}=(\pi^r_0,\pi^r_1,\dots,\pi^r_{|\rho^r|})\in\mathbb{R}^{|\rho^r|+1}$ according to \eqref{eqn:pi} for each path $\rho^r$, the following two-sided inequalities are valid for TEP for any $\hat{k}_\rho^r$:
\begin{flalign}\scriptsize  - \min \{ \pi _0^n\}_{n = 1}^m &\le \sum\limits_{(i,j) \in \rho^r} {\pi _j^r} \tilde{P}^0_{ij,k}
\le \min \{ \pi _0^n\} _{n = 1}^m \;\; \text{for } r=1,\dots,m. 
\end{flalign}
\end{lemma}

\begin{proof}
Since paths $\rho^r$ and $\rho^{r'}$ share the same starting/ending buses, this gives that $\theta_{i^r_{|\rho^r|}}-\theta_{i^r_0}=\theta_{i^{r'}_{|\rho^{r'}|}}-\theta_{i^{r'}_0}$, or equivalently, with $k$ defined as in Lemma 1,
\begin{flalign*}
\sum\limits_{(i,j) \in \rho^r} \pi_j^r {\tilde{P}}^0_{i^r_{j-1}i^r_j,k} = \sum\limits_{(i,j) \in \rho^{r'}}\pi_j^{r'} {\tilde{P}}^0_{ij,k} \;\;\;
\text{for } 1 \le r \text{, } r' \le m\:{\rm{ }} \text{with}\,\,r \ne r'
\end{flalign*}
according to the respective telescoped bus angle-difference equations of the starting and ending buses associated with each path (e.g., see \eqref{eqn:VI-E-pt1}). Thus, the proof is completed by joining together the two-sided inequalities,
\begin{flalign*}
  -\pi^r_0\le\sum_{(i,j) \in \rho^r}\pi^r_j \tilde{P}^0_{ij, k} \le\pi^r_0 \;\;\; \text{for } r=1,\dots,m,
\end{flalign*}
each of which is valid due to Lemma 1. ${\qed}$
\end{proof}

Continuing the example from subsection 4.1, in Figure \ref{fig:fig2}, $\rho^1$ creates an established parallel path with $\rho^2$. Assuming that path $\rho^2$ is the path with lower capacity-reactance product creates the example type 2 two-sided VI:\vspace{3pt}

$\begin{array}{l}
-{{\bar P}_{{i_0},{i_4}}}\,{x_{{i_0},{i_4}}} - {{\bar P}_{{i_4},{i_5}}}\,{x_{{i_4},{i_5}}} \\
\le {P_{{i_0},{i_1}}}\,{x_{{i_0},{i_1}}} + {P_{{i_1},{i_2}}}\,{x_{{i_1},{i_2}}}
 + {P_{{i_2},{i_5}}}\,{x_{{i_2},{i_5}}} \\
 \le \,{{\bar P}_{{i_0},{i_4}}}\,{x_{{i_0},{i_4}}} + {{\bar P}_{{i_4},{i_5}}}\,{x_{{i_4},{i_5}}} 
\end{array}$ 

\subsection{Parallel Paths over Established and Expansion Corridors}

Consider two buses, $\theta_n$ and $\theta_m$, in a network.  Let $\mathcal{C}$ denote the set of all paths starting at $\theta_n$ and ending at $\theta_m$.  For any path $\rho_r \in \mathcal{C}$, let $CR(\rho_r)= \sum_{(i,j) \in \rho_r} x_{ij} \bar{P}_{ij}$ denote the cumulative capacity-reactance product of one line from each corridor along that path.  Let $\bar{\rho}$ denote a path from this set such that $\overline{CR(\rho)} = \max \{ CR(\rho_r) \}$ and $\underline{\rho}$ similarly denote a path from this set such that $\underline{CR(\rho)} = \min \{ CR(\rho_r) \}$.  Further, let $N_e \left( \rho_r \right)$ denote the number of expansion corridors in the path $\rho_r$.  Note that the theorem below is stated and proved in the context of a network that meets the assumptions of all tested instances for simplicity of presentation: namely that all candidate lines for a given corridor, $(i,j)$ have identical properties (e.g., susceptance, capacity, etc.), so that additionally we can order the candidate lines.  In other words, $y_{ij,k+1} \leq y_{ij,k}$.  However, the result can be easily generalized by considering \textit{line paths}, where the path is along individual lines rather than corridors.  The details of this generalized theorem are provided in the appendix.

\begin{theorem}
The following are valid inequalities for TEP, for all paths $\rho_r \in \mathcal{C}$: 

\begin{equation}
\vert \theta_n - \theta_m \vert \leq CR(\rho_r) + \left( \overline{CR(\rho)} - CR(\rho_r) \right) \left( N_e \left( \rho_r \right) - \sum_{(i,j) \in \rho_r} \mathbb{I}_{ij} y_{ij,1} \right),
\end{equation}
where $\mathbb{I}_{ij}$ is used as shorthand for the indicator function $\mathbb{I}{\left(\omega^0_{i_{j-1},i_j}=0\right)}$ (i.e. to identify expansion corridors).

Furthermore, let $\mathcal{C}^0 \subseteq \mathcal{C}$ denote the set of paths comprised solely of established corridors, with $\rho_r^0$ denoting an element of this set. Additionally, let $\underline{CR(\rho^0)} = \min \{ CR(\rho_r^0) \}$.  If $\mathcal{C}^0$ is nonempty, then the above inequalities can be strengthened as follows:

\begin{equation}
\vert \theta_n - \theta_m \vert \leq CR(\rho_r) + \left( \underline{CR(\rho^0)} - CR(\rho_r) \right) \left( N_e \left( \rho_r \right) - \sum_{(i,j) \in \rho_r} \mathbb{I}_{ij} y_{ij,1} \right)
\end{equation}
\end{theorem}

\begin{proof}

The telescoped bus angle-difference equation \eqref{eqn:VI-E-pt1} can be written if and only if corridors $(i_0,i_1),\dots,(i_{|\rho|-1},i_{|\rho|})$ are each serviced by transmission lines (i.e., all consecutive bus-pairs must be connected).  We then consider two cases: either there are no expansion corridors in $\rho_r$ (or all expansion corridors in $\rho_r$ have at least one candidate line built) or there is at least one expansion corridor in $\rho_r$ with no candidate line built.

\textbf{Case 1: } There are no expansion corridors in $\rho_r$, or all expansion corridors in $\rho_r$ have at least one candidate line built.

Since $\mathbb{I}_{ij,1}=0$ indicates that there are existing lines servicing corridor $(i,j)$, the equation $\left( N_e \left( \rho_r \right) - \sum_{(i,j) \in \rho_r} \mathbb{I}_{ij} y_{ij,1} \right) = 0$ holds if and only the path $\rho_r$ consists entirely of serviced corridors,  that is, for any expansion corridor in the path $\rho_r$, at least one candidate line has been built.  In this case, the arguments from Lemma 1 hold for the path $\rho_r$, and we have that $| \theta_n - \theta_m | \leq \pi_0$. However, note that $\pi_0 = \sum_{(i,j) \in \rho_r}x_{ij} \overline{P}^0_{ij} = CR(\rho_r)$, so in fact we have $| \theta_n - \theta_m | \leq CR(\rho_r)$, for all $r$.

\textbf{Case 2: } There is at least one expansion corridor in $\rho_r$ with no candidate line built.

In this case, we have $\left( N_e \left( \rho_r \right) - \sum_{(i,j) \in \rho_r} \mathbb{I}_{ij} y_{ij,1} \right) \geq 1$.  Then in all cases, the inequality $| \theta_n - \theta_m | \leq \overline{CR(\rho)}$ holds.  That is, the bus angle-difference between $\theta_n$ and $\theta_m$ is bounded by the largest possible cumulative capacity-reactance product along any path between those buses.  Additionally, if $\theta_n$ and $\theta_m$ are connected by any path $\rho_r$ \textit{along established corridors}, then by Lemma 1 we again have $| \theta_n - \theta_m | \leq \pi_0 = CR(\rho_r)$.  In fact, by Lemma 2, given any collection of alternative directed paths, $\{ \rho^1, \dots, \rho^r \}$, we have $| \theta_n - \theta_m | \leq \min \{ \pi _0^k\} _{k = 1}^r$.  Similarly to case 1, note that by selecting $\{ \rho^1, \dots, \rho^r \}$ to be all paths solely along established corridors from $\theta_n$ to $\theta_m$, $\min \{ \pi _0^k\} _{k = 1}^r = \underline{CR(\rho^0)}$, thus the inequality $| \theta_n - \theta_m | \leq \underline{CR(\rho^0)}$ holds.  $\qed$
\end{proof}

As an example of the new valid inequalities described in this theorem, consider again Figure \ref{fig:fig2}.  In this figure, we consider the single paths $\rho^3 := (i_0, i_1), (i_1,i_2), (i_2,i_5), (i_5,i_6), (i_6,i_3)$ and $\rho^4 := (i_0,i_4), (i_4,i_5), (i_5,i_6), (i_6,i_3)$. Additionally, a new single path, $\rho^5$ is created when line $(i_2, i_3)$ is added where $\rho^5 := (i_0, i_1), (i_1,i_2), (i_2,i_3)$ which creates the following VI:  \vspace{3pt}

$\begin{array}{l}
\displaystyle \vert \theta_0 - \theta_3 \vert \leq \sum_{(i,j) \in \rho^5} \bar{P}_{ij} x_{ij} + \\ 
\quad\quad\quad\quad\;\; \left( \min \left\{ \sum_{(i,j) \in \rho^3} \bar{P}_{ij} x_{ij}, \sum_{(i,j) \in \rho^4} \bar{P}_{ij} x_{ij} \right\} - \sum_{(i,j) \in \rho^5} \bar{P}_{ij} x_{ij} \right) \left( 1-y_{i_2,i_3} \right)
\end{array}$

One important result to note about this theorem is how the coefficients on the right hand side relate to the $M_{ij}$ values in \eqref{eq_disyuntivo_r_6}.  Case 1 can be seen as simply summing those constraints in \eqref{eq_disyuntivo_r_6} for each $(i,j) \in \rho_k$, using the best calculated values of big-\textit{M} as described in section 2 (that is, either by a shortest path problem if bus $i$ is connected to bus $j$ or a longest path problem otherwise).  Case 2 allows the conditional use in this summation of the tighter big-\textit{M} calculated by a shortest path problem, \textit{if} enough candidate lines have been built to connect bus $i$ and bus $j$.  This is what permits these VIs to provide a strictly smaller relaxed solution space.

To illustrate the potential of these VIs, consider Figure 1 again but with the line connecting bus $i_1$ to bus $i_2$ as a candidate line instead of an existing line.  Then one VI provided by this theorem is 
\begin{align}
\vert \theta_{i_2} - \theta_{i_0} \vert &\leq 2x \bar{P} + (3x \bar{P} - 2x \bar{P} ) (1 - y_{i_1i_2}) \nonumber\\
\Rightarrow \vert \theta_{i_2} - \theta_{i_0} \vert &\leq 2x \bar{P} + x \bar{P}(1 - y_{i_1i_2}) \label{eqn:VI_effect}
\end{align}

By comparison, the best constraints (including linear combinations of constraints) relating these two buses in the original TEP model are
\begin{align}
&\vert \theta_{i_2} - \theta_{i_0} \vert \leq 4x \bar{P} \label{eqn:VI_base1}\\
&\vert \theta_{i_2} - \theta_{i_1} \vert \leq 4x \bar{P}(1-y_{i_1i_2}) \label{eqn:VI_base2}\\
&\vert \theta_{i_1} - \theta_{i_0} \vert \leq x \bar{P} \label{eqn:VI_base3}\\
&\vert \theta_{i_2} - \theta_{i_0} \vert \leq x \bar{P} + 4x \bar{P}(1 - y_{i_1i_2}) \label{eqn:VI_effect2},
\end{align}
where \eqref{eqn:VI_base1}-\eqref{eqn:VI_base3} are directly from \eqref{eq_disyuntivo_r_5} and \eqref{eq_disyuntivo_r_6}, and \eqref{eqn:VI_effect2} is the sum of \eqref{eqn:VI_base2} and \eqref{eqn:VI_base3}.

If $\mathcal{P}_{LR}$ is the polytope of the linear relaxation of the original TEP model for this simple network, and $\mathcal{P'}_{LR}$ is the polytope of the linear relaxation of the original TEP model together with \eqref{eqn:VI_effect}, then obviously $\mathcal{P'}_{LR} \subseteq \mathcal{P}_{LR}$.  In fact, it can be demonstrated that $\mathcal{P'}_{LR} \subset \mathcal{P}_{LR}$. To find a solution in $\mathcal{P'}_{LR}$, but not in $\mathcal{P}_{LR}$, assume without loss of generality that $\theta_{i_1} \geq \theta_{i_2}$, then the following system of inequalities relating bus angles $\theta_{i0}$ and $\theta_{i2}$ (which represents a point satisfying \eqref{eqn:VI_base1} and \eqref{eqn:VI_effect2} but violating \eqref{eqn:VI_effect}) must be satisfied:
\begin{align}
\theta_{i_2} - \theta_{i_0} &> 2x\bar{P} + x\bar{P} (1- y_{i_1i_2})  \label{ill}\\
\theta_{i_2} - \theta_{i_0} &\leq 3x \bar{P} \label{ill2}\\
\theta_{i_2} - \theta_{i_0} &\leq x\bar{P} + 4x\bar{P}(1-y_{i_1i_2}) \label{ill3}.
\end{align}
By joining the right hand sides of \eqref{ill} with \eqref{ill3} and \eqref{ill2} and \eqref{ill3}, this system must then satisfy
\begin{align}
2x\bar{P} + x\bar{P} (1- y_{i_1i_2}) &< 3x \bar{P} \label{eqn:star}\\
2x\bar{P} + x\bar{P} (1- y_{i_1i_2}) &< x \bar{P} + 4x \bar{P}(1-y_{i_1i_2}) \label{eqn:star2}
\end{align}
It is easy to see that \eqref{eqn:star} is true when $y_{i_1i_2} > 0$ and \eqref{eqn:star2} is true when 
\begin{align}
&\;\;\;\;x \bar{P} < 3x \bar{P}(1-y_{i_1i_2}) \nonumber\\
\Rightarrow &\;\;\;\;\;\;\;1 < 3(1-y_{i_1i_2}) \nonumber\\
\Rightarrow &\;y_{i_1i_2} < 2/3. \nonumber
\end{align}
That is, \eqref{eqn:star} and \eqref{eqn:star2} are both satisfied when $0 < y_{i_1i_2} < 2/3$.  For example, the point $y_{i_1i_2} = 0.5$, $\theta_{i_2} = 2.75 x \bar{P}$, $\theta_{i_0} = 0$ satisfies inequalities \eqref{ill} - \eqref{ill3}, i.e., it is in $\mathcal{P}_{LR}$ but not $\mathcal{P'}_{LR}$.

\section{Tests and Results}

The structure of the experiment and its implementation are as follows:  First, the low-effort heuristic method, explained in section 3, is applied.  The solution flows from the chosen relaxations are then analyzed on the same graph to find single paths of same-direction flows of maximum length using a breadth-first search algorithm. For larger instances, the maximum length of each path and maximum number of paths starting from each bus are capped to prevent memory issues. Paths with the same initial and final bus are combined to form parallel paths.  Once all or, in the case of the particularly large instances, the maximum allowed number of single paths and parallel paths are found, cuts are added to the model from those lists in a random order -- the order cuts are added has an effect on CPLEX's built-in heuristics and can change the solution time.  It should be noted that in each of the tested instances, all candidate lines for a given corridor, $(i,j)$, have identical properties (e.g., susceptance, capacity, cost, etc.).  When this is the case, we can enforce the additional set of symmetry-breaking constraints $y_{ij,k+1} \leq y_{ij,k}$, $\forall k \in \{ 1 \dots \bar{\omega}_{ij}-1\}$, since each line is interchangeable.  First, testing is performed on a version of the IEEE 118-bus system modified from \cite{christie2000power} in order to showcase the potential for the effectiveness of the proposed path-based VIs in a relatively simple and easily replicable context.  This system is also used to detail the distribution of cuts applied from each theorem.  Then, testing is performed on the Polish 2383-bus system in order to show their effectiveness in a more realistically sized and designed instance.  The algorithm is implemented in Python and solves the disjunctive model using CPLEX version 12.8.0.0.  All tests are run on the ASU High Performance Computing Agave Cluster, which has compute nodes with two Intel Xeon E5-2680 v4 CPUs running at 2.40 GHz.

\subsection{IEEE 118-Bus System} 

This system is relatively simple to solve, and in fact the unmodified 118-bus instance already satisfies demand without constructing any additional lines.  To tailor this instance for TEP and add some computational difficulty, we consider the possibility that up to 7 candidate lines with the same characteristics as the existing lines in that corridor may be added, similar to what is done with the Southern Brazilian 46-bus system in \cite{Laura1}.  Additionally, all line ratings have been reduced to 60\% in order to create congestion in the original network, as in \cite{zhang2013transmission}.  Finally, we have chosen at random 30 existing lines to remove from the system.  This allows for the theorem to be effectively applied to this test case.  The resulting system has 118 buses, 54 generators, and 186 corridors, of which 156 possess existing lines, allowing up to 7 candidate lines to be built per corridor which results in 1302 binary decisions.

We tested adding the valid inequalities to the model as initial constraints as well as with CPLEX's \textit{user cut} and \textit{lazy constraint} options \citep{cplusr}.  This was to address the fact that adding a large number of cuts as linear constraints directly via CPLEX produced inconsistent and occasionally large solution times, likely due to rounding errors \citep{escobedo2016foundational}. The number of such cuts added for both the 118-bus and 2383-bus systems is shown in detail in Table \ref{table:num_cuts}, for all potential combinations of relaxation models and broken down according to which theorem was used to generate the cut. In this table as in future tables, TR refers to the transportation relaxation, HR to the hybrid relaxation, and LR to the linear relaxation.  The user cut option allows CPLEX to implement only those inequalities it deems most beneficial at each node of the branch-and-bound process \citep{ostrowski2012tight}.  Lazy constraints behave similarly, but are only added at nodes in which a solution is found which violates those constraints, branch-and-cut style.  Although these options could increase computation time, they produced more consistent results from repeated trials and also demonstrated improvements in overall solve time.

\begin{table}[ht]
\caption{Distribution of Cuts Based on Selected Relaxations }
\label{table:num_cuts}
\begin{center}
\begin{tabular}{|c|c|c|}
\hline
Relaxation Models & 118-bus Cuts & 2383-bus Cuts \\ \hline
 TR  & 310 & 337 \\  \hline
 HR  & 15 & 835 \\  \hline
 LR  & 29 & 468\\  \hline
 TR $\oplus$ HR  & 14 & 380 \\  \hline
 TR $\oplus$ LR  & 21 & 376 \\  \hline
 HR $\oplus$ LR  & 10 &  496 \\  \hline
 TR $\oplus$ HR $\oplus$ LR  & 9 & 389 \\  \hline

\end{tabular} 
\end{center}
\end{table}

Table \ref{table:118 complete constraint} summarizes the complete results from adding all possible VIs to the 118-bus instance as full constraints and then solving. Similarly, Table \ref{table:118 complete user} does the same when adding the VIs as CPLEX user cuts.  These two options are presented as results since they proved more effective than the lazy constraint option in this instance.  For these table and for future tables, N/A refers to the time spent solving the model with no VIs added.  Additionally, the column Relax Time refers to the total time spent solving the subset of relaxation models, the column Path Search refers to the total time finding all paths and parallel paths after overlaying the solutions of the relaxation models onto the network, and the column Solution refers to the time spent solving the original problem after adding all possible VIs.  Finally, the C+P+R column is the total time spent on this whole process.  Note that all entries refer to the average runtime in seconds.

\begin{table}[ht]
\caption{IEEE 118-Bus Results Adding VIs As Constraints}
\label{table:118 complete constraint}
\begin{center}
\begin{tabular}{|c|rrrr|}
\hline
Relaxation & \multicolumn{4}{c|}{Average Computation Times (s)} \\
   
Models&Relax Time&Path Search&Solution&C+P+R\\  \hline
TR & 7.85 & 0.16 & 736.36 & 744.38 \\  \hline
HR & 11.56 & 0.15 & 392.20 & \bf 403.91 \\  \hline
LR & 0.34 & 0.13 & 799.56 & 800.03 \\  \hline
TR $\oplus$ HR & 19.00 & 0.13 & 927.17 & 946.29 \\  \hline
TR $\oplus$ LR & 8.19 & 0.12 & 1118.82 & 1127.13 \\  \hline
HR $\oplus$ LR & 11.76 & 0.14 & 783.64 & 795.53 \\  \hline
TR $\oplus$ HR $\oplus$ LR & 38.05 & 0.12 & 906.34 & 944.51 \\  \hline
 N/A	&-\;\;\;&-\;\;& 1702.08 &-\;\;\;\;\\	\hline 				

\end{tabular}  
\end{center}
\end{table}

\begin{table}[ht]
\caption{IEEE 118-Bus Results Adding VIs As User Cuts}
\label{table:118 complete user}
\begin{center}
\begin{tabular}{|c|rrrr|}
\hline
Relaxation & \multicolumn{4}{c|}{Average Computation Times (s)} \\
   
Models&Relax Time&Path Search&Solution&C+P+R\\  \hline
TR & 7.90 & 0.16 & 857.94 & 866.00\\ \hline 
HR & 11.47 & 0.15 & 851.10 & 862.72\\ \hline 
LR & 0.37 & 0.13 & 852.04 & \bf 852.53\\ \hline 
TR $\oplus$ HR & 19.06 & 0.13 & 855.24 & 874.42\\ \hline 
TR $\oplus$ LR & 8.42 & 0.12 & 866.48 & 875.03\\ \hline 
HR $\oplus$ LR & 11.62 & 0.14 & 865.50 & 877.26\\ \hline 
TR $\oplus$ HR $\oplus$ LR & 37.59 & 0.12 & 866.61 & 904.31\\ \hline
 N/A	&-\;\;\;&-\;\;& 1702.08 &-\;\;\;\;\\	\hline 				

\end{tabular}  
\end{center}
\end{table}

We can see from these tables that solving the modified 118-bus instance without adding any VIs took on average 1702.08 seconds.  In comparison, the best average total solve time including finding and implementing all VIs took 403.91 seconds, roughly a 4.2x improvement.  This time comes from solving only the hybrid relaxation and implementing the VIs as constraints directly in the model before solving in CPLEX. When adding VIs this way, the fastest individual total solve time was 306.77 (a roughly 5.6x improvement) seconds, but the slowest was 1340.21 seconds (only a 1.3x improvement). However, note that when adding the VIs as user cuts, the solution times do not vary much across all combination of relaxation models.  This is true even across individual solves: the corresponding fastest and slowest individual total solves with user cuts were 848.65 seconds and 906.74.  So while the best case improvement implementing the VIs as user cuts was only approximately 2x, the worst case performed similarly.  Thus for the 118-bus instance, user cuts do not have as much potential to improve solution times as direct constraints, but they do produce improvements with less variance.  

\subsection{Polish 2383-Bus System}

We use the Polish 2383-bus system adapted for TEP in \cite{min18sol}. This system has 2383 buses, 327 generators, and 2896 total corridors.  This system has been modified as follows: while the original 2383-bus system has candidate lines along established corridors, these options were removed and 120 of the existing lines have been removed and replaced with one candidate line each, while the remaining 2776 corridors do not allow for any expansion.  This modified instance is available upon request from the corresponding author.  Due to the size of this instance, additional testing restrictions were introduced. Limits were placed on the path-finding algorithm, permitting only 1000 paths to be found per starting bus and allowing only paths of 20 buses or fewer in length.  Although there are fewer binary variables in this instance than in the modified 118-bus system described above, only 120 in all, it is much more complex.  

\begin{table}[ht]
\caption{2383-Bus Results Adding VIs As Constraints}
\label{table:2383 constraint}
\begin{center}
\begin{tabular}{|c|rrrr|}
\hline
Relaxation & \multicolumn{4}{c|}{Average Computation Times (s)} \\
   
Models&Relax Time&Path Search&Solution&C+P+R\\  \hline
TR & 1.51 & 111.53 & 1881.91 & 1994.94  \\ \hline
HR & 31.75 & 107.83 & 2149.84 & 2289.42  \\ \hline
LR & 2.65 & 107.42 & 2931.46 & 3041.53  \\ \hline
TR $\oplus$ HR & 18.05 & 87.02 & 2663.78 & 2768.85  \\ \hline
TR $\oplus$ LR & 4.18 & 67.26 & 2167.01 & 2238.45  \\ \hline
HR $\oplus$ LR & 35.05 & 91.84 & 2773.98 & 2900.87  \\ \hline
TR $\oplus$ HR $\oplus$ LR & 36.30 & 58.72 & 1589.51 & \bf 1684.53  \\ \hline	
 N/A	&-\;\;\;&-\;\;\;& 5671.30 &-\;\;\;\;\;\\ \hline

\end{tabular}  
\end{center}
\end{table}

\begin{table}[ht]
\caption{2383-Bus Results Adding VIs As Lazy Constraints}
\label{table:2383 lazy constraint}
\begin{center}
\begin{tabular}{|c|rrrr|}
\hline
Relaxation & \multicolumn{4}{c|}{Average Computation Times (s)} \\
   
Models&Relax Time&Path Search&Solution&C+P+R\\  \hline
TR & 1.55 & 111.69 & 3279.26 & 3392.50  \\ \hline
HR & 32.51 & 108.49 & 3275.41 & 3416.41  \\ \hline
LR & 2.63 & 108.38 & 3270.65 & 3381.66  \\ \hline
TR $\oplus$ HR & 18.11 & 87.14 & 3278.83 & 3384.09  \\ \hline
TR $\oplus$ LR & 4.12 & 67.16 & 3281.24 & \bf 3352.53  \\ \hline
HR $\oplus$ LR & 34.95 & 91.55 & 3281.00 & 3407.50  \\ \hline
TR $\oplus$ HR $\oplus$ LR & 35.82 & 59.71 & 3282.46 & 3377.99  \\ \hline
 N/A	&-\;\;\;&-\;\;\;& 5671.30 &-\;\;\;\;\;\\ \hline

\end{tabular}  
\end{center}
\end{table}

As illustrated in Table \ref{table:2383 constraint}, solving any subset of relaxation models and adding a number of VIs generated from overlaying their solutions measurably reduces the average time spent solving TEP for the 2383-bus system.  Due to the time restrictions implemented in the path-finding algorithm, this is only a subset of all possible paths from which to generate VIs; however, we remark that the addition of all such VIs may be impractical from a computational standpoint, due to the exponential growth in the number of possible paths on which to base them.  In this case, solving all three relaxations produced the greatest reduction in both total solution time and in solution time not including time spent searching for paths and solving relaxations.  While the original instance took 5671.30 seconds to solve on average without adding any valid inequalities, this greatest reduction took only 1684.53 seconds, which is approximately a 3.4x speed up.  Note that in the case when only the linear relaxation solved, the equivalent average time was 3041.53 seconds, the slowest of all options and only a 1.9x improvement.  This suggests that solving multiple relaxation models, rather than just the traditional linear relaxation, can produce significant improvements to solution algorithms for TEP.  These results show the effectiveness of the proposed VIs even on systems of large size.  However, note that as in the case of the 118-bus instance, the best individual solve time was a substantial improvement over these averages.  In fact, that time was 730.54 seconds, a 7.8x improvement.  In one case out of seventy, these cuts showed no improvement, while in all other cases they showed at least a 1.8x improvement, similar to the 118-bus case.

As illustrated in Table \ref{table:2383 lazy constraint}, results from another option for adding the VIs is presented.  However, in this case these results show solution times when they are added as lazy constraints, since for the 2383-bus instance the user cut option showed the least improvements, in contrast to the 118-bus case.  That said, the effects of the alternative option are similar: best-case improvements are reduced in favor of increased consistency of improvements.

\section{Conclusions and Future Work}

This work presents a new mathematical framework and an algorithm that uses a mixed-integer linear programming model, valid inequalities, and a low-effort heuristic method for solving TEP. The objective is to reduce the total computational effort of planning.  This work is a significant improvement of the preliminary studies carried out in \cite{Laura1}, in which the solutions were found after manual analysis of the test system, creation of cuts using two classes of the valid inequalities introduced in this paper (specifically from Lemmas 1 and 2), which at that time had been implemented without proof, and tests made with different cut combinations. However, this work automates each step of the process and formally establishes the validity of three types of valid inequalities.  

Computational tests show the effectiveness of the presented theorem in generating valid inequalities which reduce the solution time of TEP, up to an 8x improvement.  They also suggest how to best apply the theorem for use in solving multiple test cases, as well as how they may be of use in the solution of larger scale problems.  Additionally, the results demonstrate different options for the implementation of these valid inequalities that offer distinct trade-offs in efficiency in the various stages of the solution process, which provides options for approaching instances of varying sizes and expected computational effort.  

In future work, we will perform a polyhedral study on the strength of the proposed VIs and we will conduct further studies to determine the most effective use of the presented theorems for particular instances. In particular, as the size of a system increases, the number of possible paths, and thus the number of possible valid inequalities, increases at an exponential rate. Finding and adding all these inequalities takes significant computational time, and the sheer number added does not necessarily improve the performance of solving via CPLEX.  Thus, additional testing is planned to determine how to select an ideal subset of single path and parallel path inequalities to help decrease total solution time, particularly in large systems.  Fine tuning of the implementation, including in regards to the optimal use of user cuts and lazy constraints, such as this will allow us to solve more complex problems, such as the L-1 reliability on TEP \citep{Lauraepec} and planning with uncertainty due to renewables as well as incorporating new technology such as FACTS devices \citep{sahraei2015fast}.

\section{Appendix}

In order to generalize Theorem 1, we introduce new definitions.  Given a path $\rho_k$, a \textit{line path} $\ell_k$ is a sequence of exactly one line per corridor $(i,j) \in \rho_k$.  The $k^{th}$ line in corridor $(i,j)$ will be denoted $(i,j,k)$ for the purposes of a line path.  For example, in a network with 3 lines per corridor, the simple path $\rho =(1,2), (2,3)$ might have line paths $\ell_1 = (1,2,1), (2,3,3)$, $\ell_2 = (1,2,3), (2,3,3)$, or $\ell_3 = (1,2,2), (2,3,2)$.  That is, $\ell_1$ is comprised of the first line from corridor $(1,2)$ and the third line from corridor $(2,3)$.  In this basic case, there are 9 possible such line paths corresponding to the path $\rho$.  Additionally, an \textit{established line path} is a line path composed entirely of existing lines, hence it corresponds to a path composed of only established paths. Let $\mathcal{C}_\ell$ be the set of all line paths.  Let $N_{e}(\ell_k)$ denote the number of candidate lines in the line path $\ell_k$, when $N_e$ is applied as a function to a line path instead of a path.  Let $\mathbb{I}_{ijk}$ represent the indicator function for candidate lines (i.e., $\mathbb{I}_{ijk} = 1$ means that line $(i,j,k)$ is a candidate line).  Given the above definitions and notations, the theorem below follows immediately from Theorem 1. 

\begin{theorem}
The following are valid inequalities for TEP, for all line paths $\ell_k \in \mathcal{C}_\ell$:
\begin{equation}
\vert \theta_n - \theta_m \vert \leq CR(\ell_k) + \left( \overline{CR(\ell)} - CR(\ell_k) \right) \left( N_e \left( \ell_k \right) - \sum_{(i,j,r) \in \ell_k} \mathbb{I}_{ijr} y_{ij,r} \right).
\end{equation}
Furthermore, let $\mathcal{C}^0 \subseteq \mathcal{C}$ denote the set of paths comprised solely of established corridors, with $\ell_k^0$ denoting an element of this set. Additionally, let $\underline{CR(\ell^0)} = \min \{ CR(\ell_k^0) \}$.  If $\mathcal{C}^0$ is nonempty, then the above inequalities can be strengthened as follows:

\begin{equation}
\vert \theta_n - \theta_m \vert \leq CR(\ell_k) + \left( \underline{CR(\ell^0)} - CR(\ell_k) \right) \left( N_e \left( \ell_k \right) - \sum_{(i,j,r) \in \ell} \mathbb{I}_{ijr} y_{ij,r} \right)
\end{equation}
\end{theorem}

\bibliographystyle{ijocv081}
\bibliography{References}

\end{document}